\documentclass[a4paper,12pt,reqno]{amsart}

\usepackage{amsmath}
\usepackage{amscd}
\usepackage{amssymb}
\usepackage{mathrsfs}
\usepackage[left=2.5cm,right=2.5cm,bottom=3cm,top=3cm]{geometry}
\newtheorem{thm}{Theorem}[section]

\newtheorem{disc}[thm]{Discussion}

\newtheorem{lem}[thm]{Lemma}

\newtheorem{ques}[thm]{Question}
\newtheorem{prop}[thm]{Proposition}
\theoremstyle{definition}
\newtheorem{defn}[thm]{Definition}
\newtheorem{rem}[thm]{Remark}

\numberwithin{equation}{section}

\begin{document}

\title[Holomorphic sectional curvature, nefness \& Miyaoka-Yau]{Holomorphic sectional curvature, nefness and Miyaoka-Yau type inequality}

\author{Yashan Zhang}
\thanks{The author is partially supported by Fundamental Research Funds for the Central Universities (No. 531118010468) and National Natural Science Foundation of China (No. 12001179)}

\address{School of Mathematics and Hunan Province Key Lab of Intelligent Information Processing and 
Applied Mathematics, Hunan University, Changsha 410082, China}
\email{yashanzh@hnu.edu.cn}

\begin{abstract}
On a compact K\"ahler manifold, we introduce a notion of almost nonpositivity for the holomorphic sectional curvature, which by definition is weaker than the existence of a K\"ahler metric with semi-negative holomorphic sectional curvature. We prove that a compact K\"ahler manifold of almost nonpositive holomorphic sectional curvature has a nef canonical line bundle, contains no rational curves and satisfies some Miyaoka-Yau type inequalities. In the course of the discussions, we attach a real value to any fixed K\"ahler class which, up to a constant factor depending only on the dimension of manifold, turns out to be an upper bound for the nef threshold.
\end{abstract}

\maketitle

\section{Introduction}
\subsection{Background} Compact K\"ahler manifolds of nef canonical line bundle (i.e. minimal models) are of great importance and interest from the viewpoint of algebraic geometry (in particular, for the minimal model program). In complex geometry, one of the fundamental goals is to give geometric criteria for positivity (e.g. nefness or ampleness) of the canonical line bundle. According to conjectures of S.-T. Yau, the negativity of the holomorphic sectional curvature should closely affect the positivity of the canonical line bundle. Precisely, Yau conjectured that a compact K\"ahler manifold $(X,\omega)$ of negative holomorphic sectional curvature has an ample canonical line bundle. This conjecture was solved in a breakthrough of Wu-Yau \cite{WY1} provided that the manifold is projective (for the surface case and threefold case, it was previously proved by B. Wong \cite{W} and Heier-Lu-B. Wong \cite{HLW}, respectively); Tosatti-Yang \cite{TY} extended Wu-Yau's work to the K\"ahler case and hence proved Yau's conjecture in full generality. Diverio-Trapani \cite{DT} and Wu-Yau \cite{WY2}  further proved that a compact K\"ahler manifold $(X,\omega)$ of quasi-negative holomorphic sectional curvature has ample canonical line bundle (see P. Wong-Wu-Yau \cite{WWY} for the special case that the manifold has Picard number one). One can also find a K\"ahler-Ricci flow approach in Nomura \cite{No} for results of Wu-Yau \cite{WY1} and Tosatti-Yang \cite{TY}. On the other hand, thanks to Yau's Schwarz Lemma \cite{Y78}, a compact K\"ahler manifold $(X,\omega)$ of semi-negative holomorphic sectional curvature does not contain any rational curve and hence, if $X$ is projective, its canonical line bundle must be nef by Mori's Cone Theorem (see e.g. \cite{Ma}); in general if $X$ is K\"ahler, then its canonical line bundle is also nef, thanks to a recent work of Tosatti-Yang \cite[Theorem 1.1]{TY}. 

\subsection{Motivations}
In general, for a given positivity concept (e.g. ampleness or nefness) of the canonical line bundle, one may naturally wonder what is the ``weakest/optimal" negativity condition on the holomorphic sectional curvature which implies the asserted positivity concept of the canonical line bundle? In contrast to the ampleness of the canonical line bundle (which is equivalent to the existence of \emph{\textbf{one}} K\"ahler \emph{\textbf{metric}} with negative Ricci curvature), the nefness of the canonical line bundle is defined by taking limits of \emph{\textbf{sequences}} of K\"ahler \emph{\textbf{classes}}, and hence it seems not ``optimal" to derive the nefness of the canonical line bundle from the semi-negativity of the holomorphic sectional curvature of one K\"ahler metric. Then it may be natural to ask:

\begin{ques}\label{ques}
For a compact K\"ahler manifold, can we find some natural (negativity) properties in terms of the holomorphic sectional curvature, which are weaker than the existence of a K\"ahler metric with semi-negative holomorphic sectional curvature but can still guarantee the nefness of the canonical line bundle?
\end{ques}
In this paper, we shall propose such a property. 
\subsection{An almost nonpositivity notion for the holomorphic sectional curvature} We first associate a real number to a K\"ahler class as follows.

\begin{defn}\label{defn1}
Let $X$ be a compact K\"ahler manifold and $\alpha$ a K\"ahler class on $X$. We define a number  $\mu_\alpha$ associated with $\alpha$ in the following way: \\
\centerline{$\mu_\alpha:=\inf\{\sup_XH_\omega|\omega$ is a K\"ahler metric in $\alpha\}$.}
\end{defn}

For a compact K\"ahler manifold $(X,\omega)$, $\sup_XH_{\omega}$ means the maximal value of the holomorphic sectional curvature of $\omega$ on $X$. To see the well-definedness of $\mu_\alpha$ in the above definition, in Section \ref{mu} we will show the following

\begin{prop}[=Proposition \ref{prop.mu}]\label{prop1.3}
For any K\"ahler class $\alpha$ on $X$, we have $\mu_\alpha>-\infty$.
\end{prop}

Obviously, $\mu_\alpha<0$ if and only if there exists a K\"ahler metric $\omega\in\alpha$ of negative holomorphic sectional curvature; however, we may observe that the condition \emph{$\mu_\alpha=0$ is by definition a condition weaker than the existence of a K\"ahler metric $\omega\in\alpha$ of semi-negative holomorphic sectional curvature}.
\par The number $\mu_{\alpha}$ has some basic properties, e.g. $\mu_{c\alpha}=c^{-1}\mu_{\alpha}$ for any positive constant $c$, $\mu_{\alpha|_{Y}}\le\mu_{\alpha}$ for any compact complex submanifold $Y$ of $X$, and $\mu_{\alpha}$ is invariant under biholomorphisms, namely, given a biholomorphism $\pi:Y\to X$ between two compact K\"ahler manifolds and a K\"ahler class $\alpha$ on $X$, we have $\mu_{\pi^*\alpha}=\mu_\alpha$. 
\par Then we define an almost nonpositivivity notion for the holomorphic sectional curvature as follows.

\begin{defn}\label{defn}
Let $(X,\hat\omega)$ be a compact K\"ahler manifold. We say that $X$ is of \textbf{almost nonpositive holomorphic sectional curvature} if for any number $\epsilon>0$, there exists a K\"ahler class $\alpha_\epsilon$ on $X$ such that $\mu_{\alpha_\epsilon}\alpha_\epsilon<\epsilon[\hat\omega]$ (i.e. $-\mu_{\alpha_\epsilon}\alpha_\epsilon+\epsilon[\hat\omega]$ is a K\"ahler class).
\end{defn}

\begin{rem}\label{rem1.4}
We remark that the almost nonpositivity notion for holomorphic sectional curvature defined in Definition \ref{defn} is a notion at the level of \textbf{$(1,1)$-classes}, not $(1,1)$-forms. 
\end{rem}
Obviously, a compact K\"ahler manifold admitting a K\"ahler metric with negative or semi-negative holomorphic sectional curvature must be of almost nonpositive holomorphic sectional curvature.

\subsection{Nefness of the canonical line bundle} The first main result is a generalization of Tosatti-Yang's result \cite[Theorem 1.1]{TY} to compact K\"ahler manifolds of almost nonpositive holomorphic sectional curvature.
\begin{thm}\label{mainthm}
A compact K\"ahler manifold of almost nonpositive holomorphic sectional curvature has a nef canonical line bundle.
\end{thm}

\begin{rem}
The Definition \ref{defn} of the almost nonpositivity for the holomorphic sectional curvature is mainly motivated by the definition of nefness of the canonical line bundle (also see Question \ref{ques}). 
\begin{itemize}
\item[(1)] Recall that, by definition, the canonical line bundle of a compact K\"ahler manifold $X$ is nef if $c_1(K_X)$ is a limit of a sequence of K\"ahler classes; in particular, nefness is a positivity at the level of $(1,1)$-classes, not $(1,1)$-forms (compare Remark \ref{rem1.4}). 
\item[(2)] Moreover, by the Calabi-Yau theorem \cite{Y} we know that the nefness of $K_X$ is equivalent to the fact that for any number $\epsilon>0$, there exists a K\"ahler metric $\omega_\epsilon$ such that $Ric(\omega_\epsilon)<\epsilon\hat\omega$, which may be seen as an almost nonpositivity notion for the Ricci curvature. 
\end{itemize}
From the above viewpoints, Definition \ref{defn} appears to be a very natural generalization of the nonpositivity property of the holomorphic sectional curvature that still guarantees the nefness of the canonical line bundle, and Theorem \ref{mainthm} in some sense says that \emph{the almost nonpositivity of the holomorphic sectional curvature implies the almost nonpositivity of the Ricci curvature}, providing an implicit relation between holomorphic sectional curvature and Ricci curvature in a certain approximate sense.
\end{rem}

\begin{rem}
In Theorem \ref{mainthm} we generalize Tosatti-Yang's result \cite[Theorem 1.1]{TY} in the sense of weakening the curvature condition on K\"ahler manifolds. In another direction, as proposed in \cite[Remark 1.7]{TY}, it is natural to generalize \cite{WY1,WY2,TY,DT} to Hermitian manifolds; and there are also progresses in this direction, see \cite{YZ,Ta}. 
\end{rem}

Our proof for Theorem \ref{mainthm} will be given in Section \ref{proof1}, which is based on an effective estimate for a lower bound of the existence time of the K\"ahler-Ricci flow in terms of the upper bound for the holomorphic sectional curvature of the initial metric. As a byproduct of our arguments, we also obtain a lower bound for the blow up rate of the maximum of holomorphic sectional curvature along the K\"ahler-Ricci flow with finite-time singularities, see Proposition \ref{prop4} in Section \ref{proof1}.

\subsection{$\mu_\alpha$ and nef thresholds} We would like to mention that, in the course of the proofs for Proposition \ref{prop1.3} (see \eqref{mu.nu1}) and Theorem \ref{mainthm} (see \eqref{ineq2.5}), some interesting relations between the number $\mu_\alpha$ and the \emph{nef threshold} of $\alpha$ are established. Recall the nef threshold $\nu_\alpha$ of $\alpha$ is defined by $\nu_\alpha:=\inf\{s\in\mathbb R|2\pi c_1(K_X)+s[\alpha]$ is nef$\}$. Obviously, the nef threshold is negative (resp. nonpositve) if and only if $K_X$ is ample (resp. nef). Then we have the following comparisons, which may be regarded as \emph{quantitative} results on aforementioned Yau's conjecture that the negativity of the holomorphic sectional curvature should closely affect the positivity of the canonical line bundle.
\begin{prop}
Let $X$ be a compact K\"ahler manifold of $dim_{\mathbb C}=n$ and $\alpha$ an arbitrary K\"ahler class on $X$.  
\begin{itemize}
\item[(1)] If $\mu_\alpha< 0$, then $\nu_\alpha\le\frac{n+1}{2n}\mu_\alpha$;
\item[(2)] If $\mu_\alpha\ge 0$, then $\nu_\alpha\le n\mu_\alpha$. 
\end{itemize} 
\end{prop} 
Note that $\mu_\alpha<0$ implying $\nu_\alpha< 0$ was exactly Yau's conjecture proved in \cite{WY1,TY}; then item (1) here may be regarded as a \emph{quantitative} version of their results. 

\subsection{A rigidity theorem: non-existence of rational curves} A classical result of Yau \cite{Y78} and Royden \cite{R} proved the non-existence of rational curves on compact K\"ahler manifolds admitting a K\"ahler metric with semi-negative holomorphic sectional curvature, which may be regarded as a rigidity theorem of negatively curved manifolds. Our second result extends this result to compact K\"ahler manifolds of almost nonpositive holomorphic sectional curvature.

\begin{thm}\label{mainthm_2}
A compact K\"ahler manifold of almost nonpositive holomorphic sectional curvature does not contain any rational curves.
\end{thm}

\begin{rem}
On the one hand, our proof for Theorem \ref{mainthm} does not depend on Theorem \ref{mainthm_2}; on the other hand, if $X$ is a projective manifold or a compact K\"ahler manifold of dimension $\le3$, then Cone Theorem holds \cite{Ma,HP} and Theorem \ref{mainthm_2} implies Theorem \ref{mainthm}. 
\end{rem}

\begin{rem}\label{rem_1}
There are many progresses on the structure and classification of compact K\"ahler manifolds of semi-negative holomorphic sectional curvature, see \cite{HLW2,HLWZ}. In our case, a natural question is the classification of compact K\"ahler manifolds of almost nonpositive sectional curvature, in which Theorem \ref{mainthm_2} may be useful. For example, as a consequence of Theorem \ref{mainthm_2}, \emph{if a smooth minimal model of general type is of almost nonpositive sectional curvature, then it must has ample canonical line bundle}, see \cite[Theorem 2]{Ka} (also see \cite[Lemma 2.1]{DT}, \cite[Lemma 5]{WY1}).
Moreover, by the Enriques-Kodaira classification on compact complex surfaces (see \cite[Chapter VI]{BHPV}), Kodaira's table of singular fibers of minimal elliptic surfaces (see \cite[Chapter V]{BHPV}) and Theorems \ref{mainthm} and \ref{mainthm_2}, \emph{a compact K\"ahler surface $X$ of almost nonpositive sectional curvature must be one of the followings: (1) complex torus or hyperelliptic surface, if $kod(X)=0$; (2) minimal elliptic surface with only singular fibers of type $mI_0$, if $kod(X)=1$; (3) K\"ahler surface with ample canonical line bundle, if $kod(X)=2$.} For K\"ahler surfaces of semi-negative holomorphic sectional curvature, one can find a complete structure theorem in \cite[Theorem 1.10]{HLW2}.
\end{rem}

Our proof for Theorem \ref{mainthm_2}, which will be given in Section \ref{proof2}, is achieved by using a positive lower bound for the maximum of holomorphic sectional curvature in terms of the existence of a rational curve. This is essentially due to Tosatti-Y.G. Zhang \cite[Proposition 1.4, Remark 4.1]{ToZyg}, where they proved a similar result for holomorphic \emph{bisectional} curvature.\\

\subsection{Miyaoka-Yau type inequalities} Finally, we observe that the Miyaoka-Yau inequality for Chern numbers holds on compact K\"ahler manifolds of almost nonpositive holomorphic sectional curvature. 
\begin{thm}\label{mainthm_3}
Assume $X$ is an $n$-dimensional ($n\ge2$) compact K\"ahler manifold of almost nonpositive holomorphic sectional curvature. Then the Miyaoka-Yau inequality holds:
\begin{equation}\label{MY1}
\left(\frac{2(n+1)}{n}c_2(X)-c_1(X)^2\right)(-c_1(X))^{n-2}\ge0.
\end{equation}
\end{thm}

It is known that the above Miyaoka-Yau inequality \eqref{MY1} holds in many cases; an incomplete list: surfaces of general type \cite{M77,Y77}, $K_X$ is ample \cite{Y77}, minimal manifolds of general type \cite{Ts1,Zyg,SWa}, minimal projective varieties \cite{GKPT,GT}  and compact K\"ahler manifolds whose $c_1(K_X)$ admits a smooth semi-positive representative \cite{No18}. 

\par Recall that for an $n$-dimensional compact K\"ahler manifold $X$ with nef canonical line bundle $K_X$, the \emph{numerical Kodaira dimension} $\nu(X)$ of $X$ is defined to be 
$$\nu(X):=\max\{k\in\{0,...,n\}|c_1(K_X)^k\neq[0]\}.$$
Then if $\nu(X)<n-2$, the conclusion in Theorem \ref{mainthm_3} holds trivially. To make a nontrivial conclusion, it is natural to replace $(-c_1(X))^{n-2}$ by $(-c_1(X))^{\nu(X)}\wedge\alpha^{n-\nu(X)-2}$ for some K\"ahler/nef class $\alpha$ on $X$ (compare e.g. \cite[Theorem B]{GT}). Motivated by these, we obtain the following result.
\begin{thm}\label{mainthm_4}
Assume $X$ is an $n$-dimensional compact K\"ahler manifold ($n\ge3$) with nef $K_X$ and the numerical Kodaira dimension $\nu=\nu(X)<n-2$.
\begin{itemize}
\item[(1)] If there exists a K\"ahler class $\alpha$ on $X$ with $\mu_\alpha=0$, then there holds
\begin{equation}\label{MY2}
\left(\frac{2(n+1)}{n}c_2(X)-c_1(X)^2\right)\cdot(-c_1(X))^{\nu}\cdot\alpha^{n-\nu-2}\ge0.
\end{equation}
\item[(2)] Let $\alpha_i, i=1,2,\ldots,$ be a sequence of K\"ahler classes on $X$ satisfying $\mu_{\alpha_i}>0$ and $\mu_{\alpha_i}\alpha_i\to0$ (i.e. $X$ is of almost nonpositive holomorphic sectional curvature). If $\alpha_i\to\alpha_\infty$ for some nef class $\alpha_\infty$, then there holds
\begin{equation}\label{MY3}
\left(\frac{2(n+1)}{n}c_2(X)-c_1(X)^2\right)\cdot(-c_1(X))^{\nu}\cdot\alpha_\infty^{n-\nu-2}\ge0.
\end{equation}
\end{itemize}
\end{thm}

\begin{rem}
For convenience, let's look at a special case of above results: assume $(X,\omega)$ is a compact K\"ahler manifold of semi-negative holomorphic sectional curvature, then Theorem \ref{mainthm_3} and Theorem \ref{mainthm_4} (1) apply, and we have
$$\left(\frac{2(n+1)}{n}c_2(X)-c_1(X)^2\right)\cdot(-c_1(X))^{i}\cdot[\omega]^{j}\ge0,$$
where $i=n-2,j=0$ if $\nu\ge n-2$, and $i=\nu,j=n-\nu-2$ if $\nu<n-2$.
\end{rem}

\par Our proofs for Theorems \ref{mainthm_3} and \ref{mainthm_4} will make use of Wu-Yau's continuity equation \cite[equations (3.2)]{WY1} to construct certain family of K\"ahler metrics, see Section \ref{proof3} for details.\\

As we mentioned before, if $K_X$ is semi-positive, i.e. the canonical class $-c_1(X)$ admits a smooth semi-positive representative, then the Miyaoka-Yau inequality \eqref{MY1} holds on $X$, thanks to a recent work of Nomura \cite[theorem 1.1]{No18}. Similar to the above Theorem \ref{mainthm_4}, if $\nu<n-2$, it seems natural to extend \cite[Theorem 1.1]{No18} to certain conclusions similar to the above inequalities \eqref{MY2} and \eqref{MY3}. Here we observe a conclusion of such type under a stronger assumption.

\begin{thm}\label{mainthm_5}
Assume $X$ is an $n$-dimensional compact K\"ahler manifold ($n\ge3$) with semi-ample canonical line bundle $K_X$ and the numerical Kodaira dimension $\nu=\nu(X)<n-2$. Then for any nef class $\alpha$ on $X$ there holds
\begin{equation}\label{MY4}
\left(\frac{2(n+1)}{n}c_2(X)-c_1(X)^2\right)\cdot(-c_1(X))^{\nu}\cdot\alpha^{n-\nu-2}\ge0.
\end{equation}
\end{thm}
Here we have assumed semi-ampleness of $K_X$, which by definition is stronger than semi-positivity of $K_X$, as well as nefness of $K_X$ (but the Abundance Conjecture predicts semi-ampleness of $K_X$ is equivalent to nefness of $K_X$). Also note that when $K_X$ is semi-ample, the numerical Kodaira dimension $\nu(X)$ equals to the Kodaira dimension of $X$. 
\par We should mention that, for a minimal \emph{projective} manifold $X$ with $\nu(X)<n-2$ and an arbitrary \emph{rational} K\"ahler class $\alpha$ on $X$, the inequality \eqref{MY4} is always true, see \cite[Theorem B]{GKPT}; then our Theorem \ref{mainthm_5} may be regarded as a partial extension of this result to the K\"ahler case. It is natural to expect that Theorem \ref{mainthm_5} should be extended to any minimal compact K\"ahler manifold with $\nu(X)<n-2$.
\par Since a nef class is a limit of a sequence of K\"ahler classes, to prove Theorem \ref{mainthm_5} it suffices to check \eqref{MY4} for any K\"ahler class $\alpha$ on $X$, which can be achieved by certain arguments due to Y.G. Zhang \cite{Zyg} using the K\"ahler-Ricci flow, see Section \ref{proof3} for details. \\

\subsection{Organization} In Section \ref{pre} a useful result of Royden will be recalled. In Section \ref{mu} we will show that for any K\"ahler class $\alpha$, $\mu_\alpha$ in Definition \ref{defn1} is well-defined. Theorems \ref{mainthm} and \ref{mainthm_2} will be proved in Sections \ref{proof1} and \ref{proof2}, respectively; Theorems \ref{mainthm_3}, \ref{mainthm_4} and \ref{mainthm_5} will be proved in Section \ref{proof3}. In Section \ref{properties}, we will present several remarks on almost nonpositive holomorphic sectional curvature. In particular we will explain why the compact K\"ahler manifolds of almost nonpositive holomorphic sectional curvature may be far from admitting a K\"ahler metric of semi-negative holomorphic sectional curvature (see Discussion \ref{disc}).

\section{Preliminaries}\label{pre}

We first recall a result of Royden \cite[Lemma, page 552]{R}:

\begin{prop}\label{prop1}\cite{R}
Let $(X,\hat\omega)$ be a K\"ahler manifold and $\kappa_{\hat\omega}(x):=\sup_{\xi\in T_x^{1,0}X\setminus\{0\}}H_{\hat\omega}(x,\xi)$. Then for any K\"ahler metric $\omega$ on $X$ we have
\begin{itemize}
\item[(1)] if $\kappa_{\hat\omega}(x)\le A$ for some number $A>0$, then at $x$,
$$g^{\bar ji}g^{\bar qp}\hat R_{i\bar jp\bar q}\le A(tr_{\omega}\hat\omega)^2;$$
\item[(2)] if $\kappa{\hat\omega}(x)\le A$ for some number $A\le0$, then at $x$, 
$$g^{\bar ji}g^{\bar qp}\hat R_{i\bar jp\bar q}\le A\frac{n+1}{2n}(tr_{\omega}\hat\omega)^2.$$
\end{itemize}
Here $\omega=\sqrt{-1}g_{i\bar j}dz^i\wedge d\bar z^j$ and $\hat R$ is the curvature tensor of $\hat\omega$.
\end{prop}

\begin{proof}
For convenience we recall here a proof due to Royden. We only check item (1). It suffices to show that, for any given point $x\in X$ and $\xi_1,\ldots,\xi_n$ a basic of $T^{1,0}_x(X)$ which is orthonormal with respect to $\omega$, there holds
\begin{equation}\label{R}
\sum_{1\le\alpha,\beta\le n}\hat R(\xi_\alpha,\bar\xi_\alpha,\xi_\beta,\bar\xi_\beta)\le A\left(\sum_{1\le\alpha\le n}|\xi_\alpha|^2\right)^2.
\end{equation} 
Recall that \cite[page 552]{R} shows, if $\sup_XH_\omega\le A$ for any fixed $A\in\mathbb R$, there holds
\begin{equation}\label{R1}
\sum_{1\le\alpha,\beta\le n}\hat R(\xi_\alpha,\bar\xi_\alpha,\xi_\beta,\bar\xi_\beta)\le \frac{1}{2}A\left(\left(\sum_{1\le\alpha\le n}|\xi_\alpha|^2\right)^2+\sum_{1\le\alpha\le n}|\xi_\alpha|^4\right).
\end{equation} 
Therefore, when $A>0$, by using an easy inequality 
$$\sum_{1\le\alpha\le n}|\xi_\alpha|^4\le \left(\sum_{1\le\alpha\le n}|\xi_\alpha|^2\right)^2,$$ 
we immediately conclude \eqref{R} from \eqref{R1}.
\par Proposition \ref{prop1} is proved.
\end{proof}

\begin{rem}
Item (2) in Proposition \ref{prop1} is very useful in solving Yau's conjecture, see \cite{WY1,WY2,TY,DT}. However, in proving our main theorems (in which we don't assume any pointwise sign for holomorphic sectional curvature), we will need $A>0$ case, i.e. item (1) in Proposition \ref{prop1}.
\end{rem}

\section{$\mu_\alpha$ is well-defined}\label{mu}
In this section we show $\mu_\alpha$ in Definition \ref{defn1} is well-defined for any K\"ahler class $\alpha$ on $X$. To this end, we need the following

\begin{lem}\label{keylem}
Let $X$ be a compact K\"ahler manifold and $\chi$ a smooth nonpositive closed real $(1,1)$-form on $X$. 
\begin{itemize}
\item[(1)] If there exists a K\"ahler metric $\hat\omega$ on $X$ such that $\kappa_{\hat\omega}\hat\omega\le\chi$ on $X$, then $2\pi c_1(K_X)+\frac{n+1}{2n}[\chi]$ is nef.
\item[(2)] If there exists a K\"ahler metric $\hat\omega$ on $X$ such that $\kappa_{\hat\omega}\hat\omega\le\chi$ on $X$ and $\kappa_{\hat\omega}\hat\omega<\chi$ at some point $x_0\in X$, then $2\pi c_1(K_X)+\frac{n+1}{2n}[\chi]$ is nef and big.
\item[(3)] If there exists a K\"ahler metric $\hat\omega$ on $X$ such that $\kappa_{\hat\omega}\hat\omega<\chi$ on $X$, then $2\pi c_1(K_X)+\frac{n+1}{2n}[\chi]$ is a K\"ahler class.
\end{itemize}
\end{lem}

\begin{rem}
These conclusions may be regarded as certainly \emph{quantitative} or \emph{ twisted} versions of results in \cite{WY1,WY2,TY,DT}. Indeed, the proofs here are simple modifications of those in \cite{WY1,WY2,TY,DT}. 
\end{rem}

\begin{proof}
Fix a K\"ahler metric $\hat\omega$ with $\kappa_{\hat\omega}\hat\omega\le\chi$. Consider the continuity method (which is a twisted version of Wu-Yau's \cite[equations (3.2)]{WY1})
\begin{equation}\label{conti.meth.}
Ric(\omega(t))=-\omega(t)+\frac{n+1}{2n}\chi+t\hat\omega.
\end{equation}
We have
\begin{align}
\Delta_{\omega(t)}\log tr_{\omega(t)}\hat\omega&\ge\frac{1}{tr_{\omega(t)}\hat\omega}\left(g^{\bar li}g^{\bar jk}\hat g_{k\bar l}R_{i\bar j}-g^{\bar ji}g^{\bar lk}\hat R_{i\bar jk\bar l}\right)\nonumber\\
&\ge\frac{1}{tr_{\omega(t)}\hat\omega}\left(g^{\bar li}g^{\bar jk}\hat g_{k\bar l}(-g_{i\bar j}+\frac{n+1}{2n}\chi_{i\bar j}+t\hat g_{i\bar j})-\frac{n+1}{2n}\kappa_{\hat\omega}(tr_{\omega(t)}\hat\omega)^2\right)\nonumber\\
&\ge \frac{n+1}{2n}tr_{\omega(t)}(\chi-\kappa_{\hat\omega}\hat\omega)+\frac tntr_{\omega(t)}\hat\omega-1\nonumber,
\end{align}
where in the first inequality we have used an inequality of Yau \cite{Y}, in the second inequality we have used the continuity equation \eqref{conti.meth.} and Proposition \ref{prop1}(2), and in the last inequality we have used (in which the condition $\chi\le0$ is needed)
$$g^{\bar li}g^{\bar jk}\hat g_{k\bar l}\chi_{i\bar j}\ge (tr_{\omega(t)}\hat\omega)\cdot(tr_{\omega(t)}\chi).$$
Set $M(t):=\frac{\frac{n+1}{2n}tr_{\omega(t)}(\chi-\kappa_{\hat\omega}\hat\omega)}{tr_{\omega(t)}\hat\omega}$, which is a nonnegative uniformly bounded function since $0\le\chi-\kappa_{\hat\omega}\hat\omega\le C\hat\omega$ for some positive constant $C$. Then
\begin{align}\label{Sineq}
\Delta_{\omega(t)}\log tr_{\omega(t)}\hat\omega&\ge\left(M(t)+\frac tn\right)tr_{\omega(t)}\hat\omega-1.
\end{align}
Now, using the maximum principle arguments as in \cite{Y,WY1,TY} one gets that there exists a unique smooth family of K\"ahler metrics $\omega(t)$ for all $t>0$ solving the continuity method \eqref{conti.meth.}. Therefore, $2\pi c_1(K_X)+\frac{n+1}{2n}[\chi]+t[\hat\omega]$ is a K\"ahler class for all $t>0$, i.e. $2\pi c_1(K_X)+\frac{n+1}{2n}[\chi]$ is a nef class. Item (1) is proved.\\

\par Next we check item (2) by simply modifying arguments in \cite{DT}. Fix a K\"ahler metric $\hat\omega$ on $X$ such that $\kappa_{\hat\omega}\hat\omega\le\chi$ on $X$ and $\kappa_{\hat\omega}\hat\omega<\chi$ at some point $x_0\in X$. By item (1) we have a unique solution $\omega(t)$, $t\in(0,1]$, to the continuity method \eqref{conti.meth.}. Define the continuous function $M(t)$ on $X\times(0,1]$ as above. We may also choose an open neighborhood $U$ of $x_0$ and a positive constant $\delta$ such that $M(t)\ge\delta$ on $U\times(0,1]$. 
\par Next we reduce \eqref{conti.meth.} to a complex Monge-Amp\`ere equation:
\begin{equation}
(t\hat\omega-Ric(\hat\omega)+\frac{n+1}{2n}\chi+\sqrt{-1}\partial\bar\partial\varphi(t))^n=e^{\varphi(t)}\hat\omega^n,
\end{equation}
where $\omega(t)=t\hat\omega-Ric(\hat\omega)+\frac{n+1}{2n}\chi+\sqrt{-1}\partial\bar\partial\varphi(t)$ is the solution to the continuity method \eqref{conti.meth.}.
By maximum principle, we see $\varphi(t)$ is uniformly bounded from above on $X\times(0,1]$.

\textbf{Claim:} there exists a time sequence $t_i\to0^+$ such that $\sup_X\varphi(t_i)$ is uniformly bounded from below by a finite number. 
\par To see this claim, let's first set $\tilde\varphi(t):=\varphi(t)-\sup_X\varphi(t)$. Note that for some fixed large constant $A\ge1$, $\tilde\varphi(t)$ is $A\hat\omega$-plurisubharmonic function with $\sup_X\tilde\varphi(t)=0$ for all $t>0$, implying that for some time sequence $t_i\to0^+$, $\tilde\varphi(t_i)$ converges in $L^1$-topology to an $A\hat\omega$-plurisubharmonic function $\tilde\varphi_0$ which is not identically $-\infty$. Then $\int_Xe^{\tilde\varphi_0}\hat\omega^n>0$; moreover, since $\{\tilde\varphi_0=-\infty\}$ is a set of measure zero, we also have $\int_Ue^{\tilde\varphi_0}\hat\omega^n>0$. Therefore, 
\begin{align}
\frac{\int_XM(t_i)\omega(t_i)^n}{\int_X\omega(t_i)^n}&=\frac{\int_XM(t_i)e^{\varphi(t_i)}\hat\omega^n}{\int_Xe^{\varphi(t_i)}\hat\omega^n}\nonumber\\
&=\frac{\int_XM(t_i)e^{\tilde\varphi(t_i)}\hat\omega^n}{\int_Xe^{\tilde\varphi(t_i)}\hat\omega^n}\nonumber\\
&\ge\frac{\delta\int_Ue^{\tilde\varphi(t_i)}\hat\omega^n}{\int_Xe^{\tilde\varphi(t_i)}\hat\omega^n}\nonumber\\
&\to\frac{\delta\int_Ue^{\tilde\varphi_0}\hat\omega^n}{\int_Xe^{\tilde\varphi_0}\hat\omega^n}
\nonumber\\
&>0\nonumber,
\end{align}
Also, obviously $\frac{\int_XM(t_i)e^{\varphi(t_i)}\hat\omega^n}{\int_Xe^{\varphi(t_i)}\hat\omega^n}$ is uniformly bounded from above. In conclusion, there exists a constant $C\ge1$ such that for every $t_i$ there holds
\begin{equation}\label{ineq3.4}
C^{-1}\le\frac{\int_XM(t_i)e^{\varphi(t_i)}\hat\omega^n}{\int_Xe^{\varphi(t_i)}\hat\omega^n}\le C.
\end{equation}

On the other hand, by Cauchy-Schwarz inequality, $tr_{\omega(t)}\hat\omega\ge n\left(\frac{\hat\omega^n}{\omega(t)^n}\right)^{\frac1n}=ne^{-\frac{\varphi(t)}{n}}$, and so 
$$\inf_Xtr_{\omega(t)}\hat\omega\ge ne^{-\frac{\sup_X\varphi(t)}{n}}.$$
Now integrating \eqref{Sineq} with respect to $\omega(t)^n=e^{\varphi(t)}\hat\omega^n$ gives
\begin{align}
\int_Xe^{\varphi(t)}\hat\omega^n &\ge\int_X\left(M(t)+\frac{t}{n}\right)tr_{\omega(t)}\hat\omega \cdot e^{\varphi(t)}\hat\omega^n\nonumber\\
&\ge ne^{-\frac{\sup_X\varphi(t)}{n}}\cdot\int_XM(t)e^{\varphi(t)}\hat\omega^n\nonumber,
\end{align}
and hence by \eqref{ineq3.4},
$$e^{\frac{\sup_X\varphi(t_i)}{n}}\ge n\frac{\int_XM(t_i)e^{\varphi(t_i)}\hat\omega^n}{\int_Xe^{\varphi(t_i)}\hat\omega^n}\ge\delta'$$
for some constant $\delta'>0$. The Claim follows.
\par Having the above Claim, we easily conclude that, up to passing to a subsequence of $t_i$, $\varphi(t_i)$ converges in $L^1$-topology to an $A\hat\omega$-plurisubharmonic function $\varphi_0$ which is not identically $-\infty$, and 
$$\int_X\left(2\pi c_1(K_X)+\frac{n+1}{2n}[\chi]\right)^n=\lim_{i\to\infty}\int_X\omega(t_i)^n=\lim_{i\to\infty}\int_Xe^{\varphi(t_i)}\hat\omega^n=\int_Xe^{\varphi_0}\hat\omega^n>0.$$
By applying \cite[Theorem 0.5]{DP}, item (2) follows.\\

\par For item (3), one observe that under the current assumption, there holds $\omega(t)\ge\delta\hat\omega$ on $X\times(0,1]$ for some constant $\delta>0$ (applying the maximum principle in \eqref{Sineq}), and so by \cite[Theorem 0.1]{DP} $2\pi c_1(K_X)+\frac{n+1}{2n}[\chi]$ is a K\"ahler class (or one can derive uniform higher order estimates for $\omega(t)$ and then show, as $t\to0$, $\omega(t)$ subsequently converges to a smooth K\"ahler metric in the class $2\pi c_1(K_X)+\frac{n+1}{2n}[\chi]$, see \cite{Y,WY1,WY2,TY}).
\end{proof}

\begin{rem}
Consider the equation on $X$:
\begin{equation}
\left(-Ric(\hat\omega)+\frac{n+1}{2n}\chi+\sqrt{-1}\partial\bar\partial u\right)^n=e^{u}\hat\omega^n.
\end{equation}
When $2\pi c_1(K_X)+\frac{n+1}{2n}[\chi]$ is nef and big, it admits a unique solution $u$ of minimal singularities \cite{BEGZ}. Set $\omega:=-Ric(\hat\omega)+\frac{n+1}{2n}\chi+\sqrt{-1}\partial\bar\partial u$, which is a K\"ahler current on $X$ and a smooth K\"ahler metric on $Amp(2\pi c_1(K_X)+\frac{n+1}{2n}[\chi])$ (the ample locus of $2\pi c_1(K_X)+\frac{n+1}{2n}[\chi]$) and satisfies
$$Ric(\omega)=-\omega+\frac{n+1}{2n}\chi,$$ 
which holds on $X$ in the current sense and on $Amp(2\pi c_1(K_X)+\frac{n+1}{2n}[\chi])$ in the smooth sense. Therefore, items (2) and (3) somehow indicate a \emph{quantitative} relation between holomorphic sectional curvature and Ricci curvature, i.e. ``a nonpositive upper bound for holomorphic sectional curvature" (in the sense of assumptions in items (2) and (3) of Lemma \ref{keylem}) implies the same (up to a constant $\frac{n+1}{2n}$) nonpositive upper bound for Ricci curvature. 
\end{rem}

\begin{rem}
We may naturally ask: can one remove nonpositivity assumption on $\chi$ in Lemma \ref{keylem} and get the same conclusions (up to changing the constant factor $\frac{n+1}{2n}$)? 
\end{rem}

Now we are ready to prove the main result in this section.
\begin{prop}\label{prop.mu}
For any K\"ahler class $\alpha$ on $X$, $\mu_\alpha$ in Definition \ref{defn1} is well-defined, i.e. $\mu_\alpha>-\infty$.
\end{prop}
\begin{proof}
Given a K\"ahler class $\alpha$ on $X$, we may assume $\mu_\alpha<0$ (otherwise, we are done).
We let $\nu_\alpha:=\inf\{s\in\mathbb R|2\pi c_1(K_X)+s[\alpha]$ is nef$\}$. Of course $\nu_\alpha>-\infty$. For any K\"ahler metric $\omega\in\alpha$ of negative holomorphic sectional curvature, there holds $\kappa_\omega\omega\le(\sup_XH_\omega)\cdot\omega$. Then we apply Lemma \ref{keylem}(1) with $\chi=(\sup_XH_\omega)\cdot\omega$ to see that $2\pi c_1(K_X)+\frac{n+1}{2n}(\sup_XH_\omega)\cdot\alpha$ is nef, implying
$$\sup_XH_\omega\ge\frac{2n}{n+1}\nu_\alpha,$$
and so
\begin{equation}\label{mu.nu1}
\mu_\alpha\ge\frac{2n}{n+1}\nu_\alpha.
\end{equation}
This proposition is proved.
\end{proof}

\begin{rem}[$\mu_\alpha\alpha$ is ``uniformly bounded from below"]
Let $X$ be a compact K\"ahler manifold. There exists a K\"ahler class $\alpha_0$ on $X$ such that for any K\"ahler class $\alpha$, $\mu_\alpha\alpha+\alpha_0$ is a K\"ahler class. Namely, $\mu_\alpha\alpha$ is ``uniformly bounded from below" in $H^{1,1}(X,\mathbb R)$. To see this, we may assume there exists a K\"ahler class $\alpha$ of $\mu_\alpha<0$. Then $2\pi c_1(K_X)$ is a K\"ahler class and the arguments in Proposition \ref{prop.mu} imply that 
$$\mu_\alpha\alpha+\left(\frac{4\pi n}{n+1}+1\right)c_1(K_X)$$
is a K\"ahler class. Our claim follows.
\end{rem}

\begin{rem}[An alternative proof of Proposition \ref{prop.mu}]
We have an alternative proof for Proposition \ref{prop.mu}, which is somehow more direct but misses the effective comparison \eqref{mu.nu1}. Recall a well-known result of Berger (see e.g. \cite[page 189, Exercise 16]{Z}): for every K\"ahler metric $\omega\in\alpha$, its scalar curvature $S^\omega$ satisfies 
\begin{equation}
\int_{\mathbb{CP}^{n-1}}H_\omega(x)\omega_{FS}^{n-1}=\frac{2}{n(n+1)}S^\omega(x)\nonumber
\end{equation}
for every $x\in X$, here $H_\omega(x):=H_\omega(x,\cdot)$ has been regarded as a function on $\mathbb{CP}^{n-1}$. Note that
\begin{equation}
\int_{X}S^\omega\omega^n=\int_XnRic(\omega)\wedge\omega^{n-1}=-2\pi nc_1(K_X)\cdot\alpha^{n-1}\nonumber.
\end{equation}
Combining the above two identities gives ($c_n$ is a positive constant depending only on $n$)
$$\sup_XH_\omega\ge-c_n\cdot\frac{2\pi c_1(K_X)\cdot\alpha^{n-1}}{\alpha^n},$$
and so

\begin{equation}\label{lowerbd}
\mu_\alpha\ge-c_n\cdot\frac{2\pi c_1(K_X)\cdot\alpha^{n-1}}{\alpha^n},
\end{equation}
proving Proposition \ref{prop.mu}. We may also mention that up to the constant factor depending on dimension of $X$, \eqref{lowerbd} is in fact implied by our above comparison \eqref{mu.nu1}, since we easily have
$$\nu_\alpha\ge-\frac{2\pi c_1(K_X)\cdot\alpha^{n-1}}{\alpha^n}.$$
Therefore, \eqref{lowerbd} can derived without using Berger's result.
\end{rem}

\section{Nefness of the canonical line bundle}\label{proof1}
To see Theorem \ref{mainthm}, we first prove the following general proposition, which provides an effective way to estimate a lower bound for the existence time of the K\"ahler-Ricci flow in terms of the upper bound for the holomorphic sectional curvature of the initial metric. 

\begin{prop}\label{prop2}
Let $(X,\hat\omega)$ be an $n$-dimensional compact K\"ahler manifold and $A=\sup_XH_{\hat\omega}$. Assume $A>0$. Then the K\"ahler-Ricci flow running from $\hat\omega$, 
\begin{equation}\label{KRF0}
\left\{
\begin{aligned}
\partial_t\omega(t)&=-Ric(\omega(t))\\
\omega(0)&=\hat\omega,
\end{aligned}
\right.
\end{equation}
exists a smooth solution on $X\times[0,\frac{1}{nA})$.
\end{prop}
\begin{proof}
The K\"ahler-Ricci flow \eqref{KRF0} is equivalent to the following parabolic Monge-Amp\`ere equation:
\begin{equation}\label{KRF1}
\left\{
\begin{aligned}
\partial_t\varphi(t)&=\log\frac{(\hat\omega-tRic(\hat\omega)+\sqrt{-1}\partial\bar\partial\varphi(t))^n}{\hat\omega^n}\\
\varphi(0)&=0,
\end{aligned}
\right.
\end{equation}
where $\omega(t)=\hat\omega-tRic(\hat\omega)+\sqrt{-1}\partial\bar\partial\varphi(t)$ is the solution to the K\"ahler-Ricci flow \eqref{KRF0}. A direct computation gives
\begin{align}\label{trace}
(\partial_t-\Delta_{\omega(t)})tr_{\omega(t)}\hat\omega&=g(t)^{\bar ji}g(t)^{\bar lk}\hat R_{i\bar jk\bar l}-g(t)^{\bar ji}g(t)^{\bar lk}\hat g^{\bar ba}\nabla^{g(t)}_i\hat g_{k\bar b}\nabla^{g(t)}_{\bar j}\hat g_{a\bar q}\nonumber\\
&\le g(t)^{\bar ji}g(t)^{\bar lk}\hat R_{i\bar jk\bar l}\nonumber\\
&\le A(tr_{\omega(t)}\hat\omega)^2,
\end{align}
where we have used Proposition \ref{prop1}(1) in the last inequality. 
\par Now we assume the K\"ahler-Ricci flow \eqref{KRF0} exists a maximal solution for $t\in[0,T)$ and assume Proposition \ref{prop2} fails, i.e. $T<\frac{1}{nA}$. Set $M(t):=\sup_Xtr_{\omega(t)}\hat\omega$ for $t\in[0,T)$, which is a smooth positive function. By applying the maximum principle in \eqref{trace} we easily have
\begin{equation}
\partial_tM(t)\le A(M(t))^2\nonumber.
\end{equation}
When $T<\frac{1}{nA}$ and $t\in[0,T)$, we have
\begin{equation}
M(t)\le\frac{n}{1-nAt}\nonumber.
\end{equation}
Set $C_1:=\frac{n}{1-nAT}$, which is a positive constant since by assumption $T<\frac{1}{nA}$. Then we have proved that 
\begin{equation}\label{C2.1}
tr_{\omega(t)}\hat\omega\le C_1
\end{equation}
on $X\times[0,T)$. 
On the other hand, from \eqref{KRF1} we have
$$(\partial_t-\Delta_{\omega(t)})\partial_t\varphi=-tr_{\omega(t)}Ric(\hat\omega).$$
We fix a positive constant $B$ with $Ric(\hat\omega)\ge-B\hat\omega$, then by \eqref{C2.1} we have
\begin{align}
(\partial_t-\Delta_{\omega(t)})\partial_t\varphi&=-tr_{\omega(t)}Ric(\hat\omega)\nonumber\\
&\le Btr_{\omega(t)}\hat\omega\nonumber\\
&\le BC_1\nonumber,
\end{align}
from which one easily concludes that 
$$\partial_t\varphi(t)\le BC_1T$$
and so 
\begin{equation}\label{C2.2}
\omega(t)^n\le e^{BC_1T}\hat\omega^n
\end{equation}
on $X\times[0,T)$. Combining \eqref{C2.1} and \eqref{C2.2} gives
\begin{align}
tr_{\hat\omega}\omega(t)&\le\frac{1}{(n-1)!}\left(tr_{\omega(t)}\hat\omega\right)^{n-1}\frac{\omega(t)^n}{\hat\omega^n}\nonumber\\
&\le\frac{C_1^{n-1}e^{BC_1T}}{(n-1)!}\nonumber.
\end{align}
Therefore, we have
\begin{equation}\label{C2.3}
C_1^{-1}\hat\omega\le\omega(t)\le\frac{C_1^{n-1}e^{BC_1T}}{(n-1)!}\hat\omega
\end{equation}
on $X\times[0,T)$ (note that $C_1>n$ and hence $C_1^{-1}<\frac{C_1^{n-1}e^{BC_1T}}{(n-1)!}$). Having \eqref{C2.3}, we can obtain uniform higher order derivatives for $\omega(t)$ on $t\in[0,T)$ and then show, as $t\to T$, $\omega(t)\to\omega(T)$ smoothly, for some K\"ahler metric $\omega(T)$ (see \cite{C,PSS}). Consequently, one can solve the K\"ahler-Ricci flow for $t\in[0,T+\epsilon)$, where $\epsilon$ is some positive number. This is a contradiction, since we have assumed that $T$ is the maximum existence time of \eqref{KRF0}. Therefore, one must have $T\ge\frac{1}{nA}$. 
\par Proposition \ref{prop2} is proved.
\end{proof}

\begin{rem}
The $A=0$ case was proved by Nomura in \cite[Proof of Theorem 1.1]{No}, which gives a K\"ahler-Ricci flow proof for \cite[Theorem 1.1]{TY}. The above Proposition \ref{prop2} extends Nomura's result \cite{No} to the general case with arbitrary positive upper bound $A$.
\end{rem}

A consequence of Proposition \ref{prop2} is the following 
\begin{prop}\label{prop2.5}
Let $X$ be an $n$-dimensional compact K\"ahler manifold and $\alpha$ K\"ahler class on $X$ of $\mu_\alpha\ge0$. Denote $\lambda_\alpha:=\sup\{t>0|\alpha+2\pi tc_1(K_X)>0\}$. Then
\begin{equation}\label{ineq2.5}
\lambda_\alpha\ge\frac{1}{n\mu_\alpha}.
\end{equation}
\end{prop}

\begin{proof}
For any small $\epsilon>0$, we choose a K\"ahler metric $\omega_\epsilon\in\alpha$ with $\sup_XH_{\omega_\epsilon}\le\mu_\alpha+\epsilon$. Then we consider the K\"ahler-Ricci flow $\omega_\epsilon(t)$ running from $\omega_\epsilon$ on $X$,
\begin{equation}\label{KRF2}
\left\{
\begin{aligned}
\partial_t\omega_\epsilon(t)&=-Ric(\omega_\epsilon(t))\\
\omega_\epsilon(0)&=\omega_\epsilon,
\end{aligned}
\right.
\end{equation}
which by Proposition \ref{prop2} has a smooth solution on $[0,\frac{1}{n(\mu_\alpha+\epsilon)})$. On the other hand, recall the K\"ahler class along the K\"ahler-Ricci flow \eqref{KRF2} is given by
$$[\omega_\epsilon(t)]=\alpha+2\pi tc_1(K_X).$$
So we have
\begin{equation}\label{2.5.1}
\lambda_\alpha\ge\frac{1}{n(\mu_\alpha+\epsilon)}.
\end{equation}
Since $\epsilon$ is an arbitrary positive constant, we conclude from \eqref{2.5.1} that
$$\lambda_\alpha\ge\frac{1}{n\mu_\alpha}.$$
\par Proposition \ref{prop2.5} is proved.

\end{proof}

Now we can give a
\begin{proof}[Proof of Theorem \ref{mainthm}]
By Definition \ref{defn} we may fix a sequence of K\"ahler classes $\alpha_i, i=1,2,\ldots,$ on $X$ such that $\mu_{\alpha_i}\alpha_i<i^{-1}[\hat\omega]$. We may assume $\mu_{\alpha_i}\ge0$ (otherwise we are done). If there exists some $i_0$ with $\mu_{\alpha_{i_0}}=0$, then by Proposition \ref{prop2.5} we see $\lambda_{\alpha_{i_0}}=\infty$ and hence $K_X$ is nef. So in the following we may assume $\mu_{\alpha_i}>0$ for all $i$. Again, by Proposition \ref{prop2.5} we have
$$\lambda_{\alpha_i}\ge\frac{1}{n\mu_{\alpha_i}}.$$
In particular,
$$[\alpha_i]+\frac{\pi}{n\mu_{\alpha_i}}c_1(K_X)$$
is a K\"ahler class, or equivalently,
$$\frac{n\mu_{\alpha_i}}{\pi}[\alpha_i]+c_1(K_X)$$
is a K\"ahler class. On the other hand, since $0<\frac{n\mu_{\alpha_i}}{\pi}\alpha_i<\frac{n}{\pi i}[\hat\omega]$, we easily have that
$$c_1(K_X)=\lim_{i\to\infty}\left(\frac{n\mu_{\alpha_i}}{\pi}[\alpha_i]+c_1(K_X)\right),$$ 
which by definition means $K_X$ is nef.
\par Theorem \ref{mainthm} is proved.
\end{proof}

We finish this section by discussing a byproduct of Proposition \ref{prop2}. The arguments for Proposition \ref{prop2} in fact prove the following:
\begin{prop}\label{prop4}
Let $X$ be an $n$-dimensional compact K\"ahler manifold and $\omega(t)$ a solution to the K\"ahler-Ricci flow \eqref{KRF0} on the maximal time interval $[0,T)$. Assume $T<\infty$. Then for every $t\in[0,T)$ we have
$$(T-t)\sup_XH_{\omega(t)}\ge \frac{1}{n}.$$
\end{prop}

\begin{proof}
Note that by the maximal existence time theorem  for the K\"ahler-Ricci flow \cite{C,Ts,TZo}, the assumption $T<\infty$ implies $K_X$ is not nef, and hence by Theorem \ref{mainthm} (or \cite[Theorem 1.1]{TY}) $\sup_XH_{\omega(t)}>0$ for every $t\in[0,T)$. Next, for any fixed $t_0\in[0,T)$, we have a K\"ahler-Ricci flow $\omega(t+t_0)$ with the maximal time interval $[0,T-t_0)$ (here we have used again the maximal existence time theorem for the K\"ahler-Ricci flow \cite{C,Ts,TZo}). Then applying Proposition \ref{prop2} gives
$$T-t_0\ge\frac{1}{n\cdot\sup_XH_{\omega(t_0)}},$$
or equivalently,
$$(T-t_0)\sup_XH_{\omega(t_0)}\ge\frac{1}{n}.$$
\par Proposition \ref{prop4} is proved.
\end{proof}

\begin{rem}
In the setting of Proposition \ref{prop4}, applying a general fact of the Ricci flow \cite[Lemma 8.6]{CLN} gives 
\begin{equation}\label{Rm}
(T-t)\sup_X|Rm(\omega(t))|\ge\frac{1}{4},
\end{equation}
which is proved by using the evolution equation of $|Rm|$ and the maximum principle (see \cite[Lemma 8.6]{CLN} for details).
On the one hand, our Proposition \ref{prop4} implies $(T-t)\sup_X|Rm(\omega(t))|\ge\frac{1}{n}$ and hence provides an alternative arguments for the above \eqref{Rm} in the K\"ahler-Ricci flow category (but the lower bound $\frac{1}{n}$ is less optimal when $n\ge5$); on the other hand, Proposition \ref{prop4} can be regarded as an improvement of \eqref{Rm} in the K\"ahler-Ricci flow category.
\end{rem}

\section{Non-existence of rational curves}\label{proof2}
To prove Theorem \ref{mainthm_2}, we need the following
\begin{prop}\cite{ToZyg}\label{prop3}
If $(X,\omega)$ is a compact K\"ahler manifold containing a rational curve $C$, then 
$$\sup_XH_\omega\ge\frac{\pi}{32\int_C\omega}.$$
\end{prop}
This is essentially due to Tosatti-Y.G. Zhang \cite[Proposition 1.4, Remark 4.1]{ToZyg}, where they proved a similar result for the maximum of holomorphic \emph{bisectional} curvature.
\begin{proof}
For convenience, we present some details by following an almost identical argument given in Tosatti-Y.G. Zhang \cite[page 2939]{ToZyg}, the only modification is replacing Yau's Schwarz lemma \cite{Y78} by the one of Royden given in Proposition \ref{prop1}(1). 
\par Set $A=\sup_XH_\omega$. Firstly, since there exists a rational curve, by Yau's Schwarz Lemma \cite{Y78} we know $A>0$. The existence of a rational curve $C$ in $X$ also implies a non-constant holomorphic map $f:\mathbb{C}\to X$. Let $\omega_{\mathbb C}$ be the Euclidean metric on $\mathbb C$. By applying a direct computation and Proposition \ref{prop1}(1) one has
\begin{equation}\label{ineq1}
\Delta_{\omega_{\mathbb C}}(tr_{\omega_{\mathbb C}}f^*\omega)\ge-A(tr_{\omega_{\mathbb C}}f^*\omega)^2.
\end{equation}
on $\mathbb C$.  Having \eqref{ineq1}, one can apply the identical arguments in \cite[page 2939]{ToZyg} to complete the proof.
\par Proposition \ref{prop3} is proved.
\end{proof}

We are ready to give a
\begin{proof}[Proof of Theorem \ref{mainthm_2}]
We prove Theorem \ref{mainthm_2} by contradiction. Assume Theorem \ref{mainthm_2} fails, i.e. there exists a rational curve $C$ in $X$. Then we apply Proposition \ref{prop3} with any K\"ahler class $\alpha$ and any K\"ahler metric $\omega\in\alpha$ to see that
$$\frac{\pi}{32\int_C\alpha}=\frac{\pi}{32\int_C\omega}\le\sup_XH_{\omega},$$
and hence 
\begin{equation}
\frac{\pi}{32\int_C\alpha}\le\mu_{\alpha}\nonumber,
\end{equation}
i.e.
\begin{equation}\label{int1}
\int_C\mu_{\alpha}\alpha\ge\frac{\pi}{32}.
\end{equation}
However, by Definition \ref{defn} we may fix a sequence of K\"ahler classes $\alpha_i, i=1,2,\ldots,$ on $X$ such that $\mu_{\alpha_i}\alpha_i<i^{-1}[\hat\omega]$, and so we easily have that
\begin{equation}\label{int2}
\limsup_{i\to\infty}\left(\int_C\mu_{\alpha_i}\alpha_i\right)\le0.
\end{equation}
By combining \eqref{int1} and \eqref{int2}, we obtain a contradiction.
\par Theorem \ref{mainthm_2} is proved.
\end{proof}

\section{Miyaoka-Yau type inequalities}\label{proof3}
In this section, we first prove Theorem \ref{mainthm_3}. 

\begin{proof}[Proof of Theorem \ref{mainthm_3}]
Note that, in view of Remark \ref{rem_1} and Yau's result \cite{Y77}, we only need to prove the case that $X$ satisfies $\int_{X}c_1(K_X)^n=0$, i.e. $K_X$ is not big. 
\par By Definition \ref{defn} we may fix a sequence of K\"ahler classes $\alpha_i, i=1,2,\ldots,$ on $X$ such that $\mu_{\alpha_i}\alpha_i<i^{-1}[\hat\omega]$. We may assume $\mu_{\alpha_i}\ge0$ for all $i$.

\par By the curvature tensor decomposition and Chern-Weil theory (see \cite[page 2752-2753]{Zyg}; also see \cite[page 96]{SWa} or \cite[Proposition 2.1]{No18}), we have, for any K\"ahler metric $\omega$ on $X$,
\begin{align}\label{CW}
\left(\frac{2(n+1)}{n}c_2(X)-c_1(X)^2\right)\cdot[\omega]^{n-2}\ge-\frac{n+2}{4\pi^2n^2(n-1)}\int_X|Ric(\omega)+\omega|_\omega^2\omega^n.
\end{align}
Therefore, to prove Theorem \ref{mainthm_3}, it suffices to construct a family of K\"ahler metrics $\tilde\omega_i$ on $X$ satisfying both of the followings as $i\to\infty$:
\begin{itemize}
\item[(1)] $[\tilde\omega_i]\to2\pi c_1(K_X)$;
\item[(2)] $\int_X|Ric(\tilde\omega_i)+\tilde\omega_i|^2_{\tilde\omega_i}\tilde\omega_i^n\to0$.
\end{itemize}

We now use Wu-Yau's continuity equation: for any given K\"ahler metric $\omega_i\in\alpha_i$, consider
\begin{equation}\label{WY}
Ric(\omega_i(t))=-\omega_i(t)+t\omega_i.
\end{equation}

By using Yau's theorem \cite{Y} and Proposition \ref{prop2.5} one easily sees that \eqref{WY} has a smooth solution $\omega_i(t)$ for $t\in(n\mu_{\alpha_i},\infty)$ (in fact, by Theorem \ref{mainthm} the solution exists for $t\in(0,\infty)$). Moreover, by Proposition \ref{prop1} and arguments in \cite[Section 2]{TY} we know, for any $\epsilon>0$, if we choose $\omega_i\in\alpha_i$ with $\sup_XH_{\omega_i}\le\mu_{\alpha_i}+\epsilon$, then the solution $\omega_i(t)$ to \eqref{WY} satisfies
\begin{equation}\label{c2}
tr_{\omega_i(n\mu_{\alpha_i}+2n\epsilon)}\omega_i\le\frac{1}{\epsilon},
\end{equation} 
which in particular implies
\begin{equation}\label{c2.1}
|\omega_i|^2_{\omega_i(n\mu_{\alpha_i}+2n\epsilon)}\le\frac{n}{\epsilon^2}.
\end{equation}
We separate discussions into two cases:
\par \emph{Case (1): there exists some $i_0$ with $\mu_{\alpha_{i_0}}=0$.} In this case, we choose a family of K\"ahler metrics $\omega_{i_0,\epsilon}\in\alpha_{i_0}$ with $\sup_XH_{\omega_{i_0,\epsilon}}\le\epsilon$ and set $\tilde\omega_\epsilon:=\omega_{i_0,\epsilon}(2n\epsilon)$, where $\omega_{i_0,\epsilon}(t)$ is the solution to 
\begin{equation}\label{WY1}
Ric(\omega_{i_0,\epsilon}(t))=-\omega_{i_0,\epsilon}(t)+t\omega_{i_0,\epsilon}.
\end{equation}
Obviously, $[\tilde\omega_\epsilon]=2\pi c_1(K_X)+2n\epsilon[\alpha_{i_0}]\to2\pi c_1(K_X)$ as $\epsilon\to0$. Moreover, since $Ric(\tilde\omega_\epsilon)+\tilde\omega_\epsilon=2n\epsilon\omega_{i_0,\epsilon}$ and \eqref{c2.1} means $|\omega_{i_0,\epsilon}|^2_{\tilde\omega_\epsilon}\le\frac{n}{\epsilon^2}$, we have
\begin{align}
\int_X|Ric(\tilde\omega_\epsilon)+\tilde\omega_\epsilon|^2_{\tilde\omega_\epsilon}\tilde\omega_\epsilon^n&=\int_X|2n\epsilon\omega_{i_0,\epsilon}|^2_{\tilde\omega_\epsilon}\tilde\omega_\epsilon^n\nonumber\\
&\le 4n^3\int_X\tilde\omega_\epsilon^n\nonumber\\
&\to 4n^3(2\pi)^n\int_X c_1(K_X)^n\nonumber\\
&=0
\end{align}
\par \emph{Case (2): $\mu_{\alpha_i}>0$ for every $i$.} In this case, we choose $\omega_i\in\alpha_i$ with $\sup_XH_{\omega_i}\le2\mu_{\alpha_i}$ and set $\tilde\omega_i:=\omega_i(3n\mu_{\alpha_i})$. Then by the Definition \ref{defn} of Property (A) $[\tilde\omega_i]=2\pi c_1(K_X)+3n\mu_{\alpha_i}\alpha_i\to2\pi c_1(K_X)$ as $i\to\infty$. Moreover, \eqref{c2.1} means $|\omega_i|^2_{\tilde\omega_i}\le\frac{n}{\mu_{\alpha_i}^2}$, and hence, as $i\to\infty$,

\begin{align}
\int_X|Ric(\tilde\omega_i)+\tilde\omega_i|^2_{\tilde\omega_i}\tilde\omega_i^n&=\int_X|3n\mu_{\alpha_i}\omega_{i}|^2_{\tilde\omega_i}\tilde\omega_i^n\nonumber\\
&\le 9n^3\int_X\tilde\omega_i^n\nonumber\\
&\to 9n^3(2\pi)^n\int_X c_1(K_X)^n\nonumber\\
&=0
\end{align}

Combining Cases (1) and (2), we have proved Theorem \ref{mainthm_3}.

\end{proof}

Next, we give a proof for Theorem \ref{mainthm_4}.

\begin{proof}[Proof of Theorem \ref{mainthm_4}] 
Let's first look at the item (1), so we have a K\"ahler class $\alpha$ with $\mu_\alpha=0$. Then by the same arguments in above proof for Theorem \ref{mainthm_3}, we can construct a sequence of K\"ahler metric $\tilde\omega_\epsilon$ for $\epsilon>0$ satisfying
\begin{itemize}
\item[(1)] $[\tilde\omega_\epsilon]=2\pi c_1(K_X)+2n\epsilon\alpha$;
\item[(2)] $\int_X|Ric(\tilde\omega_\epsilon)+\tilde\omega_\epsilon|^2_{\tilde\omega_\epsilon}\tilde\omega_\epsilon^n\le4n^3\int_X\tilde\omega_\epsilon^n$
\end{itemize}
Observe that, when $\nu<n-2$, 
\begin{align}\label{CW1}
&\frac{1}{\epsilon^{n-\nu-2}}\left(\frac{2(n+1)}{n}c_2(X)-c_1(X)^2\right)\cdot[\tilde\omega_\epsilon]^{n-2}\nonumber\\
&=\frac{1}{\epsilon^{n-\nu-2}}\left(\frac{2(n+1)}{n}c_2(X)-c_1(X)^2\right)\cdot(-2\pi c_1(X)+2n\epsilon\alpha)^{n-2}\nonumber\\
&=\frac{1}{\epsilon^{n-\nu-2}}\left(\frac{2(n+1)}{n}c_2(X)-c_1(X)^2\right)\cdot\left(\sum_{k=0}^{\nu}\binom{n-2}{k}(2\pi)^k(2n\epsilon)^{n-k-2}(-c_1(X))^k\cdot\alpha^{n-k-2}\right)\nonumber\\
&\to\binom{n-2}{\nu}(2\pi)^\nu(2n)^{n-\nu-2}\left(\frac{2(n+1)}{n}c_2(X)-c_1(X)^2\right)\cdot(-c_1(X))^k\cdot\alpha^{n-k-2},
\end{align}
as $\epsilon\to0$. On the other hand, 
\begin{align}\label{CW2}
&\frac{1}{\epsilon^{n-\nu-2}}\int_X|Ric(\tilde\omega_\epsilon)+\tilde\omega_\epsilon|^2_{\tilde\omega_\epsilon}\tilde\omega_\epsilon^n\nonumber\\
&\le\frac{4n^3}{\epsilon^{n-\nu-2}}\int_X\tilde\omega_\epsilon^n\nonumber\\
&=\frac{4n^3}{\epsilon^{n-\nu-2}}\int_X\sum_{k=0}^\nu\binom{n}{k}(2\pi)^k(2n)^{n-k}\epsilon^{n-k}(-c_1(X))^k\wedge\alpha^{n-k}\nonumber\\
&\le C\epsilon^2\to0,
\end{align}
as $\epsilon\to0$. Plugging \eqref{CW1} and \eqref{CW2} into \eqref{CW} gives the desired result \eqref{MY2}.
\par Next we look at the item (2). Obviously, we only need to discuss the case $\alpha_\infty\neq[0]$. When $\alpha_\infty\neq[0]$, we must have $\mu_{\alpha_i}\to0$ as $i\to\infty$. By the same arguments in above proof for Theorem \ref{mainthm_3}, for $\epsilon>0$ we can construct a sequence of K\"ahler metric $\tilde\omega_i$ satisfying
\begin{itemize}
\item[(1)] $[\tilde\omega_i]=2\pi c_1(K_X)+3n\mu_{\alpha_i}\alpha_i$;
\item[(2)] $\int_X|Ric(\tilde\omega_i)+\tilde\omega_i|^2_{\tilde\omega_\epsilon}\tilde\omega_i^n\le9n^3\int_X\tilde\omega_i^n$
\end{itemize}
Similar to the item (1), we can show that, as $i\to\infty$,
\begin{align}\label{CW21}
&\frac{1}{\mu_{\alpha_i}^{n-\nu-2}}\left(\frac{2(n+1)}{n}c_2(X)-c_1(X)^2\right)\cdot[\tilde\omega_i]^{n-2}\nonumber\\
&\to\binom{n-2}{\nu}(2\pi)^\nu(3n)^{n-\nu-2}\left(\frac{2(n+1)}{n}c_2(X)-c_1(X)^2\right)\cdot(-c_1(X))^k\cdot\alpha_\infty^{n-k-2},
\end{align}
and
\begin{align}\label{CW22}
&\frac{1}{\mu_{\alpha_i}^{n-\nu-2}}\int_X|Ric(\tilde\omega_i)+\tilde\omega_i|^2_{\tilde\omega_i}\tilde\omega_i^n\nonumber\\
&\le C\mu_{\alpha_i}^2\to0.
\end{align}
Plugging \eqref{CW21} and \eqref{CW22} into \eqref{CW} gives the desired result \eqref{MY3}.
\par Theorem \ref{mainthm_4} is proved.
\end{proof}

Finally, we give a
\begin{proof}[Proof of Theorem \ref{mainthm_5}]
We may assume without loss of any generality that $\alpha$ is a K\"ahler class on $X$. The proof we discuss here is a simple modification of arguments in Y.G. Zhang \cite{Zyg} (also see \cite{No18}). For an arbitrary K\"ahler class $\alpha$ on $X$ and a K\"ahler metric $\omega_0\in\alpha$, we consider the K\"ahler-Ricci flow $\omega=\omega(t)_{t\in[0,\infty)}$ running from $\omega_0$:
\begin{equation}\label{krf}
\left\{
\begin{aligned}
\partial_t\omega(t)&=-Ric(\omega(t))-\omega(t)\nonumber\\
\omega(0)&=\omega_0\nonumber,
\end{aligned}
\right.
\end{equation}
along which the K\"ahler class satisfies $[\omega(t)]=-2\pi(1-e^{-t})c_1(X)+e^{-t}\alpha$. Note that, as in the above proof for Theorem \ref{mainthm_4}, we easily have, as $t\to\infty$,
\begin{align}\label{CW31}
&e^{(n-\nu-2)t}\left(\frac{2(n+1)}{n}c_2(X)-c_1(X)^2\right)\cdot[\omega(t)]^{n-2}\nonumber\\
&e^{(n-\nu-2)t}\left(\frac{2(n+1)}{n}c_2(X)-c_1(X)^2\right)\cdot(-(1-e^{-t})2\pi c_1(X)+e^{-t}\alpha)^{n-2}\nonumber\\
&\to\binom{n-2}{\nu}(2\pi)^\nu\left(\frac{2(n+1)}{n}c_2(X)-c_1(X)^2\right)\cdot(-c_1(X))^\nu\cdot\alpha^{n-\nu-2},
\end{align}
Then, to complete the proof we make the following\\ 
\textbf{Claim}: as $t\to\infty$ there holds
\begin{equation}\label{CW32_0}
e^{(n-\nu-2)t}\int_X|Ric(\omega(t))+\omega(t)|^2_{\omega(t)}\omega(t)^n\to0.
\end{equation}
To see this, we need the followings:
\begin{equation}\label{evo1}
\partial_t\omega(t)^n=-(S(t)+n)\omega(t)^n\nonumber,
\end{equation}
\begin{equation}\label{evo2}
\partial_tS(t)=\Delta_{\omega(t)}S(t)+|Ric(\omega(t))+\omega(t)|^2_{\omega(t)}-(S(t)+n)\nonumber.
\end{equation}

By calculations in \cite{Zyg}, we have
\begin{align}
&\int_X|Ric(\omega(t))+\omega(t)|^2_{\omega(t)}\omega(t)^n\nonumber\\
&=\int_X(\partial_tS(t))\omega(t)^n+\int_X(S(t)+n)\omega(t)^n\nonumber\\
&=\partial_t\left(\int_XS(t)\omega(t)^n\right)+\int_X(S(t)+1)(S(t)+n)\omega(t)^n\nonumber,
\end{align}
and so
\begin{align}\label{CW32}
&e^{(n-\nu-2)t}\int_X|Ric(\omega(t))+\omega(t)|^2_{\omega(t)}\omega(t)^n\nonumber\\
&=e^{(n-\nu-2)t}\partial_t\left(\int_XS(t)\omega(t)^n\right)+e^{(n-\nu-2)t}\int_X(S(t)+1)(S(t)+n)\omega(t)^n\nonumber\\
&=\partial_t\left(e^{(n-\nu-2)t}\int_XS(t)\omega(t)^n\right)+e^{(n-\nu-2)t}\int_X(S(t)^2+(\nu+3)S(t)+n)\omega(t)^n.
\end{align}

Set $L(t):=e^{(n-\nu-2)t}\int_XS(t)\omega(t)^n$. By direct computation we have
\begin{align}
L(t)&=e^{(n-\nu-2)t}\int_XS(t)\omega(t)^n\nonumber\\
&=ne^{(n-\nu-2)t}(2\pi c_1)\cdot(-2\pi(1-e^{-t})c_1(X)+e^{-t}\alpha)^{n-1}\nonumber\\
&=\sum_{k=0}^{\nu-1}A_k(1-e^{-t})^ke^{(k-\nu-1)t}\nonumber,
\end{align}
where $A_k$'s are some constant only depending on $c_1(X),\alpha,n$ and $k$, and hence we can find a positive constant $C$ such that
\begin{equation}\label{CW32.0}
L(t)\le Ce^{-2t}.
\end{equation}
and
\begin{equation}\label{CW32.1}
\partial_tL(t)\le Ce^{-2t}.
\end{equation}

Using \eqref{CW32.0} and the easy fact $[\omega(t)]^n\le Ce^{-(n-v)t}$, we see that the second term in \eqref{CW32} satisfies
\begin{align}
&e^{(n-\nu-2)t}\int_X(S(t)^2+(\nu+3)S(t)+n)\omega(t)^n\nonumber\\
&=e^{(n-\nu-2)t}\int_XS(t)^2\omega(t)^n+(\nu+3)L(t)+ne^{(n-\nu-2)}[\omega(t)]^n\nonumber\\
&\le e^{(n-\nu-2)t}\int_XS(t)^2\omega(t)^n+C_2e^{-2t}\nonumber,
\end{align}
where $C$ is some uniform positive constant. 
\par Next we recall a result of Song-Tian \cite{ST} that the scalar curvature $S(t)$ is uniformly bounded on $X\times[0,\infty)$ when $K_X$ is semi-ample (this is the only place using the semi-ampleness of $K_X$), and so 
\begin{align}\label{CW32.2}
&e^{(n-\nu-2)t}\int_X(S(t)^2+(\nu+3)S(t)+n)\omega(t)^n\nonumber\\
&\le Ce^{(n-\nu-2)t}\int_X\omega(t)^n+Ce^{-2t}\nonumber\\
&\le Ce^{-2t}
\end{align}

Plugging \eqref{CW32.1} and \eqref{CW32.2} into \eqref{CW32} gives
$$e^{(n-\nu-2)t}\int_X|Ric(\omega(t))+\omega(t)|^2_{\omega(t)}\omega(t)^n\le Ce^{-2t},$$
which completes the proof of Claim. 
\par Now the \eqref{MY4} follows by plugging \eqref{CW31} and \eqref{CW32_0} into \eqref{CW}.
\par Theorem \ref{mainthm_5} is proved.
\end{proof}

\section{Remarks on the almost nonpositive holomorphic sectional curvature}\label{properties}

In this section, we shall make some more remarks on the almost nonpositive holomorphic sectional curvature introduced in Definition \ref{defn}. Let's begin with the following one, which directly follows from the well-known decreasing property of holomorphic sectional curvature on submanifolds.

\begin{prop}\label{prop_a}
Let $X$ be a compact K\"ahler manifold of almost nonpositive holomorphic sectional curvature and $Y$ a compact complex submanifold of $X$. Then $Y$ is also of almost nonpositive holomorphic sectional curvature.
\end{prop} 

Given Proposition \ref{prop_a},  all the conclusions in Theorems \ref{mainthm}, \ref{mainthm_2} and \ref{mainthm_3} hold for every compact complex submanifold of $X$, provided $X$ is of almost nonpositive holomorphic sectional curvature. For example, as in \cite[Corollary]{WY1} and \cite[Remark 1.5]{TY}, we can conclude that  \emph{if $X$ is a compact K\"ahler manifold of almost nonpositive holomorphic sectional curvature, then every compact complex submanifold of $X$ has a nef canonical line bundle.}\\

Next we observe that the almost nonpositive holomorphic sectional curvature is preserved under the product of K\"ahler manifolds.
\begin{prop}\label{prop_b}
Let $X,Y$ be two compact K\"ahler manifolds and $Z:=X\times Y$ the product manifold with the product complex structure. If both $X$ and $Y$ are of almost nonpositive holomorphic sectional curvature, then $Z$ is also of almost nonpositive holomorphic sectional curvature.
\end{prop} 
\begin{proof}
To complete the proof, we need the following \\
\textbf{Claim}: \emph{Given a K\"ahler metric $\omega$ on $X$ and a K\"ahler metric $\eta$ on $Y$, if we choose two positive constants $A_1$ and $A_2$ satisfying $\sup_XH_{\omega}\le A_1$ and $\sup_YH_{\eta}\le A_2$, then the product K\"ahler metric $\chi:=\omega+\eta$ on $Z$ satisfies $\sup_ZH_{\chi}\le A_1+A_2$.}
\par This Claim can be easily checked by using the definition of holomorphic sectional curvature and the product structure; so we omit the details here. 
\par By assumption, we fix a sequence of K\"ahler classes $\alpha_i$ on $X$ and $\beta_i$ on $Y$ such that $0\le\mu_{\alpha_i}\alpha_i\to0$ and $0\le\mu_{\beta_i}\beta_i\to0$. We separate discussions into three cases as follows:\\

\emph{Case 1: there exist a K\"ahler class $\alpha$ on $X$ with $\mu_\alpha=0$ and a K\"ahler class $\beta$ on $Y$ with $\mu_\beta=0$.} In this case we have $\mu_{\alpha+\beta}\le0$. Indeed, for any $\epsilon>0$ we choose K\"ahler metrics $\omega_\epsilon\in\alpha$ and $\eta_\epsilon\in\beta$ such that $\sup_XH_{\omega_\epsilon}\le\epsilon$ and $\sup_YH_{\eta_\epsilon}\le\epsilon$, and define $\chi_\epsilon:=\omega_\epsilon+\eta_\epsilon$ be the product K\"ahler metric on $Z$. By the above Claim we have $\sup_ZH_{\chi_\epsilon}\le2\epsilon$, and so $\mu_{\alpha+\beta}\le0$.\\

\emph{Case 2: there exists a K\"ahler class $\alpha$ on $X$ with $\mu_\alpha=0$ and $\mu_{\beta_i}>0$ for every $i$.} For a fixed sequence of positive numbers $\epsilon_i$ with $\epsilon_i\to0$, we can choose K\"ahler metrics $\omega_i\in\alpha$ with $\sup_XH_{\omega_i}\le\epsilon_i$. Next we set $\tilde\beta_i:=\frac{\mu_{\beta_i}}{\epsilon_i}\beta_i$, which is a sequence of K\"ahler metrics on $Y$ with $\mu_{\tilde\beta_i}=\epsilon_i$. Notice that $\epsilon_i\tilde\beta_i=\mu_{\tilde\beta_i}\tilde\beta_i=\mu_{\beta_i}\beta_i\to0$. Now we choose for every $i$ a K\"ahler metric $\eta_i\in\tilde\beta_i$ with $\sup_YH_{\eta_i}\le2\epsilon_i$ and define K\"ahler metrics $\chi_i:=\omega_i+\eta_i$ on $Z$. By the above Claim we know $\sup_ZH_{\chi_i}\le3\epsilon_i$, and hence
$$\mu_{\alpha+\tilde\beta_i}(\alpha+\tilde\beta_i)\le3\epsilon_i(\alpha+\tilde\beta_i)\to0.$$

\emph{Case 3: $\mu_{\alpha_i}>0$ and $\mu_{\beta_i}>0$ for every $i$.} Similar to Case 2, we may define a sequence of K\"ahler classes on $Y$ by $\tilde\beta_i:=\frac{\mu_{\beta_i}}{\mu_{\alpha_i}}\beta_i$, which satisfies $\mu_{\tilde\beta_i}=\mu_{\alpha_i}$ and $\mu_{\alpha_i}\tilde\beta_i=\mu_{\tilde\beta_i}\tilde\beta_i=\mu_{\beta_i}\beta_i\to0$. Also we can choose K\"ahler metrics $\omega_i\in\alpha_i$ and $\eta_i\in\tilde\beta_i$ with $\sup_XH_{\omega_i}\le2\mu_{\alpha_i}$ and $\sup_YH_{\eta_i}\le2\mu_{\tilde\beta_i}=2\mu_{\alpha_i}$. Therefore,
$$\mu_{\alpha_i+\tilde\beta_i}(\alpha_i+\tilde\beta_i)\le4\mu_{\alpha_i}(\alpha_i+\tilde\beta_i)\to0.$$
Proposition \ref{prop_b} is proved.
\end{proof}

\begin{disc}\label{disc}
Set $\mathcal M^n_1$ be the space of $n$-dimensional compact K\"ahler manifolds admitting a K\"ahler metric with semi-negative holomorphic sectional curvature, $\mathcal M^n_2$ the space of $n$-dimensional compact K\"ahler manifolds admitting a K\"ahler class $\alpha$ with $\mu_\alpha=0$, and $\mathcal M^n_3$ the space of $n$-dimensional compact K\"ahler manifolds of almost nonpositive holomorphic sectional curvature. Obviously we have the inclusions $\mathcal M^n_1\subset\mathcal M^n_2\subset\mathcal M^n_3$. We may have a natural question that are these inclusions strict? Let's look at the first inclusion $\mathcal M^n_1\subset\mathcal M^n_2$. For an arbitrary $X\in\mathcal M^n_2$, we fix a K\"ahler class $\alpha$ on $X$ with $\mu_\alpha=0$ and a sequence of K\"ahler metrics $\omega_i\in\alpha$ with $\sup_XH_{\omega_i}\le i^{-1}$. A natural approach to check $X\in\mathcal M^n_1$ is to prove that $\omega_i$ converges to a K\"ahler metric $\omega_\infty\in\alpha$ smoothly (or at least in $C^2$). If this can be carried out, then easily we have $H_{\omega_\infty}\le0$ and hence $X\in\mathcal M^n_1$. However, it seems unclear how to obtain uniform higher order estimates for $\omega_i$ (in part because $\omega_i$'s do not satisfy any geometric equations) and hence hard to get the smooth convergence of $\omega_i$. In general, $\omega_i$ may degenerate or blowup as $i\to\infty$. More generally, if we consider an arbitrary $X\in\mathcal M^n_3$ and fix a sequence of K\"ahler classes $\alpha_i$ on $X$ with $\mu_{\alpha_i}\alpha_i\to0$, then even the K\"ahler classes $\alpha_i$ may go to the boundary or infinity of the K\"ahler cone of $X$ (and so it is impossible to get smooth convergences to a K\"ahler metric from the family of K\"ahler metrics $\omega_i\in\alpha_i$). From this viewpoint, it seems that \emph{a compact K\"ahler manifold of almost nonpositive holomorphic sectional curvature is in general far from admitting a K\"ahler metric with semi-negative holomorphic sectional curvature}. Unfortunately, at the moment we don't have any examples of compact K\"ahler manifolds in $\mathcal M^n_3\setminus\mathcal M^n_1$. Along this line we have some natural interesting problems, e.g. (1) find examples of compact K\"ahler manifolds in $\mathcal M^n_3\setminus\mathcal M^n_1$; (2) given an $X\in\mathcal M^n_3$, find characterizations on $X$ such that $X\in\mathcal M^n_1$; more restrictly, for a given $X\in\mathcal M^n_2$ and a K\"ahler class $\alpha$ on $X$ with $\mu_\alpha=0$, find certain characterizations on $(X,\alpha)$ such that there exists a K\"ahler metric $\overline\omega\in\alpha$ with $H_{\overline\omega}\le0$. Such metric $\overline\omega$, if exists, may be naturally called an HSC-extremal metric in $\alpha$, and may be considered as a canonical/good metric in $\alpha$.
\end{disc}

\section*{Acknowledgements}
The author is grateful to Professors Huai-Dong Cao and Gang Tian for their interest in this work, and constant encouragement and support. The author also thanks Professors Gang Tian for conversations on Chern number inequality, Valentino Tosatti for discussions on papers \cite{DT,TY,WY1}, Chengjie Yu and Tao Zheng for comments and the referees for careful readings and a number of helpful comments and corrections.

\end{document}